\documentclass[11pt]{article}

\usepackage[utf8]{inputenc}
\DeclareUnicodeCharacter{0308}{\"{}}

\usepackage{titling}

\usepackage{abstract}

\usepackage{titlesec}
\def\sectionfont{\sffamily\Large\bfseries\boldmath}
\def\subsectionfont{\sffamily\large\bfseries\boldmath}
\def\paragraphfont{\sffamily\normalsize\bfseries\boldmath}
\titleformat*{\section}{\sectionfont}
\titleformat*{\subsection}{\subsectionfont}
\titleformat*{\subsubsection}{\paragraphfont}
\titleformat*{\paragraph}{\paragraphfont}
\titleformat*{\subparagraph}{\paragraphfont}
\usepackage[small,labelfont={bf,sf}]{caption}  % Change captions
\usepackage[margin=1in]{geometry} 
\date{}

\usepackage{nicefrac}       % compact symbols for 1/2, etc.
\usepackage{xcolor}         % colors

\usepackage[normalem]{ulem}
\usepackage{amsmath,amssymb,graphicx,amsthm}

\usepackage[hyphens]{url}
\newtheorem{proposition}{Proposition}
\newtheorem{assumption}{Assumption}
\newtheorem{theorem}{Theorem}
\newtheorem{lemma}{Lemma}
\newtheorem{remark}{Remark}
\DeclareMathOperator*{\argmin}{\arg\!\min}

\def\reals{\mathbb{R}}

\def\norm#1{\left\|#1\right\|}

\usepackage{todonotes}
\newcommand{\addcite}[0]{\ifthenelse{\boolean{showcomments}}
{\textcolor{purple}{(add cite(s)) }}{}}%

\usepackage{hyperref}
\usepackage{cleveref}
\usepackage{enumitem}
\usepackage{appendix}

\usepackage[
backend=bibtex,
style=numeric-comp,
sorting=none,
maxbibnames=12,
]{biblatex}
\title{\bfseries\sffamily
Safe Gradient Flow for Bilevel Optimization}

\author{
    \normalsize Sina Sharifi$^*$, Nazanin Abolfazli$^*$,
    Erfan Yazdandoost Hamedani$^\dag$, Mahyar Fazlyab$^\dag$
}
\begin{document}
\maketitle

%%%%%%%%%%%%%%%%%%%%%%%%%%%%%%%%%%%%%%%%%%%%%%%%%%%%%%%%%%%%%%%%%%%%%%%%%%%%%%%%%
\vspace{-1cm}
\begin{abstract}

Bilevel optimization is a key framework in hierarchical decision-making, where one problem is embedded within the constraints of another. In this work, we propose a control-theoretic approach to solving bilevel optimization problems. Our method consists of two components: a gradient flow mechanism to minimize the upper-level objective and a safety filter to enforce the constraints imposed by the lower-level problem. Together, these components form a safe gradient flow that solves the bilevel problem in a single loop. To improve scalability with respect to the lower-level problem's dimensions, we introduce a relaxed formulation and design a compact variant of the safe gradient flow. This variant minimizes the upper-level objective while ensuring the lower-level decision variable remains within a user-defined suboptimality. Using Lyapunov analysis, we establish convergence guarantees for the dynamics, proving that they converge to a neighborhood of the optimal solution. Numerical experiments further validate the effectiveness of the proposed approaches. Our contributions provide both theoretical insights and practical tools for efficiently solving bilevel optimization problems.
% advancing the state of the art in this field.

\end{abstract}
% \begin{keywords}%
%   Bilevel Optimization, Barrier Function, Safe Gradient Flow%
% \end{keywords}
\renewcommand{\thefootnote}{}
\footnotetext{$^*$Equal Contribution, $^\dag$Equal Advising.}
\footnotetext{S. Sharifi and M. Fazlyab are with the Department of Electrical and Computer Engineering at Johns Hopkins University, Baltimore, MD 21218, USA. \
\tt\small \{sshari12, mahyarfazlyab\}@jhu.edu}
\footnotetext{N. Abolfazli, E. Yazdandoost Hamedani are with the Department of Systems and Industrial Engineering, The University of Arizona, Tucson, AZ 85721, USA. \
\tt\small \{nazaninabolfazli, erfany\}@arizona.edu}
%%%%%%%%%%%%%%%%%%%%%%%%%%%%%%%%%%%%%%%%%%%%%%%%%%%%%%%%%%%%%%%%%%%%%%%%%%%%%%%%
\section{Introduction}
Bilevel optimization is crucial for addressing hierarchical decision-making problems that finds a plethora of applications in engineering \cite{pandvzic2018investments}, economics \cite{von1952theory}, transportation \cite{sharma2015iterative}, and machine learning \cite{finn2017model, rajeswaran2019meta,fei2006one,hong2020two,bengio2000gradient,hao2024bilevel,zhang2022revisiting}.
 A broad category of bilevel optimization problems can be written as optimization problems of the form
\begin{alignat}{2}\label{eq: bilevel task 1}\tag{BLO}
    & \min_{x \in \mathcal{X}} \ \ell(x) \! := \! f(x,y^\star(x)) \ 
    &\text{s.t.} \  y^\star(x) \! \in \! \argmin_{y \in \mathcal{C}(x)} ~g(x,y) ,%\notag \nonumber
\end{alignat}
% %
where $f,g: \mathbb R^n\times\mathbb R^m \to \mathbb R$ represent the upper-level and lower-level objective functions, and $\mathcal X\subseteq \mathbb R^n$, $\mathcal C(x)\subseteq \mathbb R^m$ are upper-level and lower-level constraint sets. Moreover, $\ell:\mathbb R^n\to \mathbb R$ denotes the implicit objective function. 
For simplicity in exposition, this paper only considers the case where both levels are unconstrained; we will defer the constrained case to future work.

Bilevel optimization problems are inherently non-convex and computationally demanding, primarily due to the interdependent coupling between the upper-level and lower-level problems. In this paper, we introduce a novel control-theoretic framework to solve bilevel problems. We start with a gradient-flow dynamical system that merely seeks to solve the upper-level problem, disregarding the constraints imposed by the lower-level problem. We will then leverage the notion of set invariance to design a safety filter, posed as a convex quadratic program (QP), that minimally modifies the dynamics of the gradient flow to efficiently enforce the constraints imposed by the optimality conditions of the lower-level optimization problem. The resulting filtered dynamical system will be \textit{anytime} safe with suitable initialization and is guaranteed to find near-optimal solutions to the bilevel problem under certain assumptions. We will establish a non-asymptotic convergence rate through Lyapunov analysis. See \Cref{fig:oveview} for an overview of the method.

The proposed safety filter requires matrix inversion. To avoid this computational burden for high-dimensional problems, we will also present a compact version of the safety filter that avoids matrix inversions altogether.  Utilizing Lyapunov analysis we prove the convergence of the dynamics to a neighborhood of the optimal solution under certain regularity assumptions.
Finally, we present numerical experiments and ablation studies to demonstrate the effectiveness of our proposed methods.

\begin{figure}[t]
    \centering
    \includegraphics[width=0.6\linewidth]{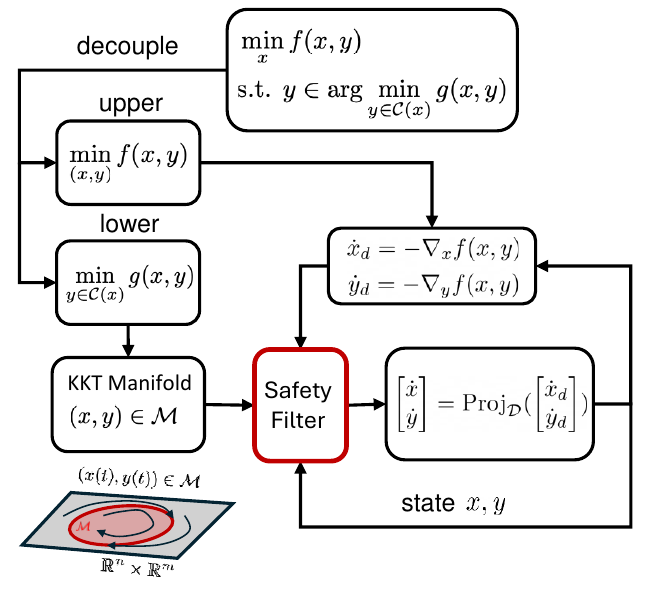}
    \caption{An overview of the proposed method. 
    The gradient flow attempts to minimize the upper-level objective, while the safety filter ensures that the constraints induced by the lower-level problem are satisfied.}
    \label{fig:oveview}
\end{figure}

\subsection{Related Work}
Bilevel optimization problems can often be reduced to single-level formulations by leveraging the optimality conditions of the lower-level problem \cite{hansen1992new}. However, large-scale or non-convex lower-level problems present significant challenges. Recent advances in gradient-based bilevel optimization methods primarily fall into two categories: approximate implicit differentiation (AID) \cite{pedregosa2016hyperparameter,domke2012generic,ghadimi2018approximation} and iterative differentiation (ITD) \cite{shaban2019truncated,Franceschi_ICML18,grazzi2020iteration}. More recent research has addressed challenges such as non-unique lower-level solutions \cite{sow2022constrained,chen2023bilevel,shen2024method,liu2022bome,xiao2023generalized} and the presence of lower-level constraints \cite{xiao2023alternating, khanduri2023linearly, xu2023efficient, kornowski2024first}.

Of particular relevance to our method, \cite{liu2022bome} employs dynamic barriers and first-order methods to solve bilevel problems. Barrier functions, widely used in control theory \cite{ames2016control, ames2019control, tan2024zero}, have also been adopted in optimization for ensuring safety \cite{muehlebach2022constraints, allibhoy2023control,allibhoy2021anytime}. The term safe gradient flow was introduced in \cite{allibhoy2021anytime, allibhoy2023control} in the context of solving nonlinear optimization problems using control barrier functions.  Our work builds on these concepts to develop a safe gradient flow framework tailored for bilevel optimization. Other approaches include orthogonal gradient projection \cite{schechtman2023orthogonal} and differential geometry-based updates \cite{tanabe1980geometric}, particularly in discrete-time settings.

While continuous-time optimization has seen extensive development, its application to bilevel problems remains less explored. \cite{chen2022single} establishes conditions for asymptotic convergence to stationary points of the upper-level problem while ensuring convergence of the lower-level solution. Two-timescale stochastic gradient descent in continuous time has been analyzed for almost sure convergence \cite{sharrock2023two}. Continuous-time methods are particularly advantageous in engineering, finance, and natural sciences, where continuous models align with high-frequency sampling \cite{yin1993continuous}. These methods also provide better approximations and mitigate issues like ill-conditioning \cite{salgado1988connection}, biased estimates \cite{chen2005least}, and divergence \cite{sirignano2017stochastic}.
\subsection{Notation}
An extended class $\mathcal{K}$ function is a function $\alpha:(-b, a) \mapsto \mathbb{R}$ for some $a, b > 0$ that is strictly increasing and satisfies $\alpha(0) = 0$. We denote $L_r^{-}(h) = \{x \in \mathbb{R}^n \mid h(x) \leq r \}$ as the $r$-sublevel set of the function $h(\cdot)$. 
We denote the Euclidean norm by $\|\cdot\|$ and $[x]_+:= \max\{0,x\}$. 
The function $h(\cdot)$ is called $L^h$-Lipschitz if for any $z, z'$ we have  $\|h(z) - h(z')\| \leq L^h \|z - z'\|$. 
Given a map  $F(x): \mathbb{R}^n  \to \mathbb{R}^m$, its Jacobian matrix is denoted by \( D_xF(x)\in\mathbb{R}^{m\times n} \).
Furthermore, we denote the Euclidean projection of $\bar{x} \in \mathbb{R}^n$ onto a closed convex set $\mathcal{D} \subseteq \mathbb{R}^n$ as $\mathrm{Proj}_\mathcal{D}(\bar{x}):=\argmin_{x\in\mathcal{D}}\|x-\bar{x}\|^2$.

\subsection{Preliminaries}
Consider the following regularity assumptions. 

\begin{assumption}\label{assumption:upperlevel}
The upper-level objective $f$ is a continuously differentiable function with bounded gradient, i.e., there exists $C_x^f,C_y^f>0$ such that $\|\nabla_x f(x,y)\|\leq C_x^f$ and $\|\nabla_y f(x,y)\|\leq C_y^f$ for all $(x,y)\in\mathbb{R}^n\times \mathbb{R}^m$.
Furthermore, $\nabla_x f(x,y)$ and $\nabla_y f(x,y)$ are $L_{x}^f$ and $L^f_{y}$-Lipschitz continuous, respectively, for $(x,y) \in \reals^n \times \reals^m$. 
\end{assumption}
\begin{assumption}\label{assumption:lowerlevel}
 The lower-level objective function $g(x,\cdot)$ is twice continuously differentiable and
 \begin{enumerate}[leftmargin=*, label=\roman*., ref=\theassumption-\roman*]
    \item\label{assumption:lowerlevel-1} $\mu_g$-strongly convex for all $x \in \reals^n$.
    \item\label{assumption:lowerlevel-2} $\nabla_y g(x,\cdot)$, and $\nabla_y g(\cdot,y)$, are $L_{yy}^g$- and $L^g_{yx}$-Lipschitz continuous for all $x\in \reals^n$, and $ y \in \reals^m$, respectively.
    \item\label{assumption:lowerlevel-3} $\nabla^2_{yx} g(x,y)$ and $\nabla^2_{yy} g(x,y)$ are $C_{yx}^{g}$ and $C_{yy}^{g}$-Lipschitz continuous, respectively, for $(x,y) \in \reals^n \times \reals^m$.
\end{enumerate}
\end{assumption}

Under Assumption \ref{assumption:lowerlevel}, the optimal solution map \( y^\star(x) \) of the lower-level problem is single-valued and continuously differentiable. This conclusion follows from the optimality condition of the lower-level problem, \( \nabla_y g(x, y^\star(x)) = 0 \) for all \( x \in \mathbb{R}^n \). Since the Hessian matrix \( \nabla_{yy}^2 g(x, y^\star(x)) \) is positive definite (and therefore non-singular), the implicit function theorem \cite{Rudin1976} guarantees that the map \( x \mapsto y^\star(x) \) is unique and continuously differentiable. Moreover, the Jacobian of the solution map, \( D_x y^\star(x) \), can be derived by differentiating the implicit equation with respect to \( x \),
\[
\nabla_{yx}^2 g(x, y^\star(x)) + \nabla_{yy}^2 g(x, y^\star(x)) D_x y^\star(x) = 0, \quad \forall x \in \mathbb{R}^n.
\]
Using this equation, the gradient of the implicit objective function $\ell(\cdot)$ (defined in \eqref{eq: bilevel task 1}), also known as the hyper-gradient, can be calculated as follows,
\begin{subequations}\label{eq:grad-implicit}
\begin{align}
    &\nabla \ell(x) = \nabla_x f(x,y^\star(x)) - \nabla_{yx}^2 g(x,y^\star(x))^\top v^\star(x), \\
    &\nabla_{yy}^2 g(x,y^\star(x)) v^\star(x) = \nabla_y f(x,y^\star(x)).
\end{align}
\end{subequations}
Calculating $\nabla \ell(x)$ requires the exact computation of $y^\star(x)$. To avoid this, it is common to consider the following surrogate map to estimate \eqref{eq:grad-implicit} by replacing the optimal lower-level solution with an approximated solution $y$,
\begin{subequations}\label{eq:grad-surrogate}
\begin{align}
    &F(x,y) :=  \nabla_x f(x,y) - \nabla_{yx}^2 g(x,y)^\top v(x,y),\\
    &\nabla_{yy}^2 g(x,y) v(x,y) = \nabla_y f(x,y).
\end{align}
\end{subequations}
Under suitable assumptions on the upper-level objective function $f$, this estimate can be controlled by the distance between the optimal solution $y^\star(x)$ and the approximated solution $y$.

\begin{lemma}\label{lem:error-grad-ell}
Under Assumptions~\ref{assumption:upperlevel} and \ref{assumption:lowerlevel}, there exists $M_1>0$ such that 
$
        \norm{\nabla \ell(x) - F(x,y)} \leq M_1 \norm{y-y^\star(x)}. \notag
    $ for all $x \in \mathbb{R}^n, y \in \mathbb{R}^m$.
\end{lemma}
\begin{proof} The proof follows that of Lemma 4.1 of \cite{abolfazli2023inexact}.
\end{proof}

\section{Safe Gradient Flow for Solving Bilevel Optimization Problems}\label{sec:Invariance}
Under Assumption \ref{assumption:lowerlevel-1} (strong convexity), the lower-level solution \( y^\star(x) \) is unique and differentiable, reducing the bilevel optimization problem to an equivalent constrained, non-convex single-level optimization problem:
\begin{alignat}{2}\label{eq: bilevel task 1 b}
    & \min_{x,y} \ f(x,y) \quad 
    &\text{s.t.} \quad (x,y) \in \mathcal{M} = \{(x,y) \mid \nabla_y g(x,y)=0\}. %\nabla_y g(x,y)=0 ,%\notag \nonumber
\end{alignat}
In our proposed design paradigm, the upper-level optimizer chooses a base dynamical system to update both upper and lower decision variables without considering the lower-level problem. Choosing the gradient flow, we can write
\begin{align}\label{eq: gradient flow upper level} \tag{GF}
    \dot{x} = - \nabla_x f(x,y), \quad
    \dot{y} = - \nabla_y f(x,y). \notag
\end{align}
This dynamical system aims to minimize the upper-level objective function while disregarding the constraints \((x, y) \in \mathcal{M}\) imposed by the lower-level optimization problem. To take the lower-level problem into account, we propose to design a ``\emph{safety filter}'' that \emph{minimally} modifies the velocities $\dot{x}$ and $\dot{y}$ such that trajectories of $x$ and $y$ are attracted to and remain on $\mathcal{M}$. More explicitly, suppose the velocity vectors $(\dot{x},\dot{y})$ are enforced to satisfy the following condition for a given $\alpha>0$,
\begin{align} \label{eq: lower level invariance unconstrained}
    \frac{d}{dt} \nabla_y g(x,y) + \alpha \nabla_y g(x,y)=0 \implies \nabla_{yx}^2 g(x,y) \dot{x} + \nabla_{yy}^2 g(x,y) \dot{y} + \alpha \nabla_y g(x,y) = 0.
\end{align}
With this affine constraint on the velocity vectors, off-manifold trajectories are \emph{contracted} towards the manifold \(\mathcal{M}\), while on-manifold trajectories remain within \(\mathcal{M}\). Together, \eqref{eq: lower level invariance unconstrained} ensures that \(\mathcal{M}\) is \emph{exponentially stable and forward invariant}.

To ensure the gradient flow dynamics \eqref{eq: gradient flow upper level} satisfy \eqref{eq: lower level invariance unconstrained}, we propose the following convex QP to compute the closest velocity to the gradient flow that guarantees \(\mathcal{M}\) is exponentially stable and forward invariant,
\begin{align}\label{eq:convex QP problem} \tag{SF}
    (\dot{x},\dot{y}) =\argmin_{(\dot{x}_d,\dot{y}_d)}~ &\frac{1}{2}\|\dot{x}_d+\nabla_x f(x,y)\|^2  + \frac{1}{2}\|\dot{y}_d+\nabla_y f(x,y)\|^2\\
    \text{s.t.}\quad &\nabla_{yx}^2 g(x,y) \dot{x}_d + \nabla_{yy}^2 g(x,y) \dot{y}_d + \alpha \nabla_y g(x,y) = 0. \notag
\end{align}
This QP admits a unique closed-form solution, which we call \emph{safe gradient flow (SGF)}.
\begin{align} \label{eq: gradient flow upper level safe} \tag{SGF}
    \dot{x} &= -\nabla_x f(x,y) - {\nabla_{yx}^2 g(x,y)^\top \lambda(x,y)} \\ \dot{y} &= -\nabla_y f(x,y) - {\nabla_{yy}^2 g(x,y)^\top \lambda(x,y)} \notag
\end{align}
Here, \(\lambda(x, y)\) is the vector of dual variables for the QP, with the following closed-form expression (notation simplified by omitting dependence on \(x\) and \(y\)),
\begin{align}\label{eq:lambda1}
\lambda(x,y) \!= \!-\! \left(\nabla_{yx}^2 g\nabla_{yx}^2 g^\top \! + \!(\nabla_{yy}^2 g)^2\right)^{-1} (\nabla_{yx}^2 g \nabla_x f\! +\! \nabla_{yy}^2 g \nabla_y f\! -\! \alpha \nabla_{y}g)
\end{align}
Equation \eqref{eq: gradient flow upper level safe} comprises two terms: a negative gradient to reduce the upper-level objective and a constraint-aware correction to ensure forward invariance of the constraint set \(\mathcal{M}\). Before proceeding with the convergence analysis of \eqref{eq: gradient flow upper level safe}, we first make a few important remarks:
\begin{enumerate}[leftmargin=*]
    \item The right-hand side of \eqref{eq: gradient flow upper level safe} is Lipschitz continuous, guaranteeing that the ODE is well-posed and admits a unique solution. (See \Cref{Appendix:add} \Cref{rm:Lip}.)
    \item The equilibrium of \eqref{eq: gradient flow upper level safe} satisfies the Karush-Kuhn-Tucker (KKT) conditions of \eqref{eq: bilevel task 1 b}, ensuring that the ODE can solve the bilevel problem as intended. (See \Cref{Appendix:add} \Cref{rm:Equil}.)
    \item Instead of projecting the decision variables \((x, y)\), \eqref{eq:convex QP problem} projects the velocity vectors. This subtle distinction is significant: velocity projection results in a convex QP even for non-convex sets (here \(\mathcal{M}\)).
    \item The QP \eqref{eq:convex QP problem} is guaranteed to be feasible if the matrix \( [\nabla_{yx} g \ \nabla_{yy} g] \) has full rank. This condition is less restrictive than Assumption \ref{assumption:lowerlevel-1}, as the strong convexity of \( g \) (i.e., \( \nabla_{yy}^2 g(x, y) \succeq \mu_g I \)) is sufficient to ensure that \( [\nabla_{yx} g \ \nabla_{yy} g] \) is full rank.
\end{enumerate} 

\begin{theorem}[Convergence of \eqref{eq: gradient flow upper level safe}]\label{thm:Original}
    Suppose Assumptions \ref{assumption:upperlevel} and \ref{assumption:lowerlevel} hold, and let $x(t),y(t)$ be the solution of \eqref{eq: gradient flow upper level safe} for $\alpha>0$. 
    Define the Lyapunov function
    $$
    \mathcal{E}(t) = f(x(t),y(t)) - f^* + \beta \|\nabla_y g(x(t), y(t))\| + c\int_{0}^t \|F(x(\tau),y(\tau))\|^2 d\tau.
    $$
   for some $\beta, c>0$, where $f^\star$ is the global optimal value of \ref{eq: bilevel task 1 b}. Then we have that 
    \begin{enumerate}
        \item The trajectories will be \emph{contracted} towards $\mathcal M$ exponentially fast,
        $$\|\nabla_y g(x(t),y(t)) \| = e^{-\alpha t} \|\nabla_yg(x(0),y(0))\|.$$
        \item There exists $\beta, c >0$ such that $\dot{\mathcal{E}}(t)  \leq  0$ for all $t\geq 0$. In particular,
        \begin{align}    
        % \frac{1}{t}\int_{0}^t \|\nabla \ell (x(\tau))\|^2 d\tau \leq &\frac{1}{ct}\big(f(x(0),y(0)) - f^* + \beta \|\nabla_y g(x(0),y(0))\|\big) \notag \\
        % &+ \frac{1 - e^{-\alpha t}}{t} \frac{M_1}{\alpha \mu_g} \|\nabla_yg(x(0),y(0))\|. \notag \\
        \frac{1}{t}\int_{0}^t \|\nabla \ell (x(\tau))\|^2 d\tau 
        \leq &\frac{2\mathcal{E}(0)}{ct}+ \frac{1 - e^{-2\alpha t}}{\alpha t} \frac{M_1^2}{\mu_g^2} \|\nabla_yg(x(0),y(0))\|^2. \notag
        \end{align}
    \end{enumerate}
\end{theorem}

\begin{proof}
    See \Cref{proof:thmSGF}.   
\end{proof}

\section{Inversion-free Safe Gradient Flow}\label{sec:InversionFree}
The QP \eqref{eq:convex QP problem} has a closed-form solution; however, computing it involves inverting an \( m \times m \) matrix (where \( m \) is the dimension of \( y \)) within the dual variable \( \lambda(x, y) \). This process becomes computationally expensive for high-dimensional lower-level problems. To mitigate the scalability challenge, we start with the following equivalent single-level reduction of the bilevel problem
\begin{alignat}{2}\label{eq: bilevel task 2 a}
    & \min_{(x,y)} \ f(x,y) \quad 
    &\text{s.t.} \quad (x,y) \in \mathcal{M} = \{(x,y) \mid  h(x,y):=\|\nabla_y g(x,y)\|^2=0\}.
\end{alignat}
This reformulation reduces \(m\) nonlinear equality constraints to a single nonlinear equality constraint. Using a similar approach as before, we can minimally perturb the gradient flow of the upper-level objective to enforce the invariance condition, yielding the following convex QP:
\begin{align}\label{eq:convex QP problem 1}
\begin{split}
    (\dot{x},\dot{y})=\argmin_{(\dot{x}_d,\dot{y}_d)}~ &\frac{1}{2}\|\dot{x}_d+\nabla_x f(x,y)\|^2  + \frac{1}{2}\|\dot{y}_d+\nabla_y f(x,y)\|^2\\
    \text{s.t.}\quad &
    \nabla_x h(x,y)^\top \dot{x}_d \!+\! \nabla_y h(x,y)^\top \dot{y}_d \!+\! \alpha h(x,y) = 0.
\end{split}
\end{align}
The feasible set of \eqref{eq:convex QP problem 1} forms a hyperplane as long as \(h(x, y) \neq 0\).  Solving the QP
, we obtain
\begin{equation} \label{eq: gradient flow upper level safe 2} 
\begin{aligned}
    \dot{x} &= -\nabla_x f(x,y) - \lambda(x,y)\nabla_{x}h(x,y), \\
    \dot{y} &= -\nabla_y f(x,y) - \lambda(x,y)\nabla_y h(x,y), \\
    \lambda(x,y)\!&=  \frac{-\nabla_x h(x,y)^\top \nabla_x f(x,y)\! -\! \nabla_{y}h(x,y)^\top \nabla_y f(x,y)\! +\! \alpha h(x,y)}{\|\nabla_{x}h(x,y)\|^2 \! + \! \|\nabla_{y}h(x,y)\|^2},
\end{aligned}
\end{equation}
where $\lambda(x,y)$ is the dual variable of \eqref{eq:convex QP problem 1}. When $h(x,y)=0$, the feasible set spans the full space \(\mathbb{R}^{m+n}\), resulting in the ODE $\dot{x}=-\nabla_x f(x,y)$ and $\dot{y}=-\nabla_y f(x,y)$. Thus, we can extend the definition of $\lambda(x,y)$ on $\mathcal{M}$ by setting $\lambda(x,y)=0$. 

The dynamics in \eqref{eq: gradient flow upper level safe 2} eliminate the matrix inversion required in \eqref{eq:lambda1}, significantly improving scalability. However, this efficiency introduces a challenge: \((\dot{x}, \dot{y})\) becomes discontinuous at points where \(h(x, y) = 0\), and \(\lambda\) becomes unbounded compared to \eqref{eq: gradient flow upper level safe}. In the next subsection, we propose an alternative approach that avoids this discontinuity entirely.

\subsection{Relaxation of the KKT Conditions}
Observing that the right-hand side of \eqref{eq: gradient flow upper level safe 2} is discontinuous on $\mathcal{M}$, we consider the following relaxation of the compact formulation in \eqref{eq: bilevel task 2 a},
\begin{align}\label{eq:bilevel-approx}
    \min_{x,y}~f(x,y)\quad 
    \text{s.t.} \quad (x,y) \in L_{\varepsilon^2}^{-}(h): = \{(x,y) \mid h(x,y) -\varepsilon^2 \leq 0 \},
\end{align}
where $\varepsilon>0$ is an arbitrary tolerance parameter. This relaxation is practical, as exact lower-level solutions are often unrealistic. Now to ensure forward invariance for the constraint set of \eqref{eq:bilevel-approx}, we  use a barrier-like inequality used in safety-critical control theory and specifically control barrier functions, e.g., \cite{ames2019control}, as follows,
    \begin{align} \label{eq: invariance condition}
    \frac{d}{dt} (h(x,y) \!-\! \varepsilon^2) \!+\! \alpha (h(x,y) \!-\! \varepsilon^2) \leq 0.
    \end{align}
    Projecting the gradient flow dynamics onto this half-space constraint, we obtain the following QP-based relaxed safety filter (RXSF),
    \begin{alignat}{2} \label{eq: first-order lower-level dynamics 2} \tag{RXSF}
    (\dot{x},\dot{y})=&\argmin_{(\dot{x}_d,\dot{y}_d)}~ &&\frac{1}{2}\|\dot{x}_d+\nabla_x f(x,y)\|^2  + \frac{1}{2}\|\dot{y}_d+\nabla_y f(x,y)\|^2\\
    &\text{s.t.}\quad && \nabla_x h(x,y)^\top \dot{x}_d \!+\! \nabla_y h(x,y)^\top \dot{y}_d \!+\! \alpha (h(x,y) - \varepsilon^2) \leq 0.\notag    
    \end{alignat}
This projection has the following closed-form solution, where $\lambda$ is  the dual variable in \eqref{eq: first-order lower-level dynamics 2} (see \Cref{derivation}),
\begin{align}\label{eq:projectionXY}\tag{RXGF}
        \begin{split}
            \dot{x} &= -\nabla_x f(x,y) - \lambda(x,y) \nabla_x h(x,y),
            \\
            \dot{y} &= -\nabla_y f(x,y) - \lambda(x,y) \nabla_y h(x,y), \\
            \lambda(x,y) \!&=\! 
            \frac{\left[-\nabla_x h^\top \nabla_x f \!-\! \nabla_y h^\top \nabla_y f \!+\! \alpha (h - \varepsilon^2)\right]_{+}}
            {\|\nabla_x h\|^2 + \|\nabla_y h\|^2}. \notag
        \end{split} 
    \end{align}
We summarize the essential properties of \eqref{eq:projectionXY} in the following proposition.
\begin{proposition}\label{thm:epsilon feasible} \normalfont
    The following statements are true for \eqref{eq: first-order lower-level dynamics 2}:
    \begin{enumerate}[label=(\roman*), ref=\theproposition-\roman*]
        \item\label{prop:Lip} The right-hand side of \eqref{eq:projectionXY} is Lipschitz continuous, ensuring the uniqueness of the solutions.
        \item\label{prop:Feas} If the initial condition satisfies \((x(0), y(0)) \in L_{\varepsilon^2}^{-}(h) \), then \( (x(t), y(t)) \in L_{\varepsilon^2}^{-}(h)\) for all \(t \geq 0\). 
        \item\label{prop:KKT} The equilibrium of \eqref{eq: first-order lower-level dynamics 2} recovers the solution of \eqref{eq:bilevel-approx}.
        \item\label{prop:feas} The QP \eqref{eq: first-order lower-level dynamics 2} is always feasible.
        \end{enumerate}
\end{proposition}
\begin{proof}
    See \Cref{proof:PropRXGF}.
\end{proof}

\subsection{Convergence Analysis} 
The dynamics in \eqref{eq: first-order lower-level dynamics 2} ensure that the trajectory \((x(t), y(t))\) remains feasible for the approximate problem \eqref{eq:bilevel-approx}. Next, we focus on proving the convergence of trajectories to the solution of \eqref{eq:bilevel-approx}. Let \(f^\star_\varepsilon\) denote the optimal value of \eqref{eq:bilevel-approx}. Consider the following function, 
\begin{align}\label{eq:lyapunov-inversion-free}
\mathcal{E}(t) \! := \!  f(x(t),y(t)) \! - \! f^\star_\varepsilon \! + c\int_{0}^t \|F(x(\tau), y(\tau))\|^2 d\tau,
\end{align}
for some $c>0$. 
In \Cref{thm:converganceRate} we show that \eqref{eq:lyapunov-inversion-free} is a Lyapunov function for \eqref{eq: first-order lower-level dynamics 2}.

\begin{theorem}[Convergence of \eqref{eq: gradient flow upper level safe 2}]\label{thm:converganceRate}
    Suppose Assumption \ref{assumption:lowerlevel} holds, and let $x(t)$ and $y(t)$ be the solutions of \eqref{eq:projectionXY}. 
    Let $c:= \mu_g^2/(\mu_g^2+(L_{yx}^g)^2)$. Then for any $t\geq 0$, $\mathcal{E}(t)  \leq  \mathcal{E}(0)$. In particular, 
    \begin{align}
        \frac{1}{t}\int_{0}^t \|\nabla \ell(x(\tau))\|^2 d\tau 
        \leq \frac{2M_1^2}{\mu_g^2}\varepsilon^2 
        + \frac{2}{ct}( f(x(0),y(0)) - f^*_{\varepsilon}). \notag
    \end{align}
\end{theorem}
\begin{proof}
    See \Cref{proof:thmRXGF}.
\end{proof}

\begin{remark}[Comparison with \cite{liu2022bome}] \normalfont
    In \cite{liu2022bome}, the authors apply the value function approach along with a dynamic barrier similar to \eqref{eq: invariance condition} to develop a first-order method for solving \eqref{eq: bilevel task 1}. 
    This approach enables them to avoid Hessian calculations, though it requires computing the lower-level optimal value $y^\star(x)$ at each iteration. To mitigate this bottleneck, they approximate $y^\star(x)$ by performing $T$ steps of gradient descent on the lower-level problem, which can compromise the guarantees of their method. In contrast, our method employs the implicit function approach, eliminating the need to compute $y^\star(x)$. This ensures convergence, as demonstrated in \Cref{thm:converganceRate}.
\end{remark}

\section{Experiments}\label{sec:experiments}
In this section, we evaluate our method's performance and analyze the impact of each hyperparameter on a synthetic problem and a data hyper-cleaning task using the MNIST dataset. All experiments employ the Fourth-order Runge-Kutta (RK-4) discretization scheme. The code is available at \href{https://github.com/o4lc/SGF-BLO}{https://github.com/o4lc/SGF-BLO}.

\begin{figure*}[t]
    % \vspace{5mm}
    \centering
    % \begin{subfigure}[t]{0.33\textwidth}
    \begin{minipage}{0.33\textwidth}
        \includegraphics[width=\textwidth]{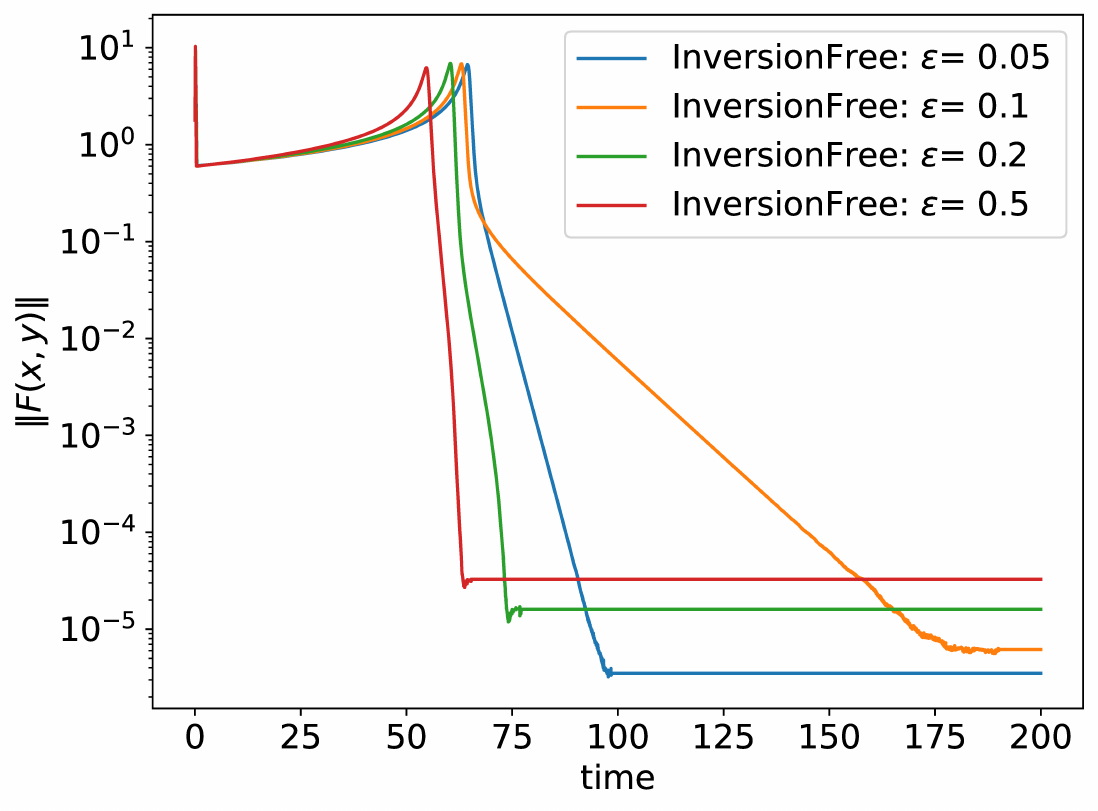}
        % \caption{}
        \label{fig:toyExample-Eps1}
    \end{minipage}
    % \end{subfigure}
    % 
    \begin{minipage}{0.32\textwidth}
        \includegraphics[width=\textwidth]{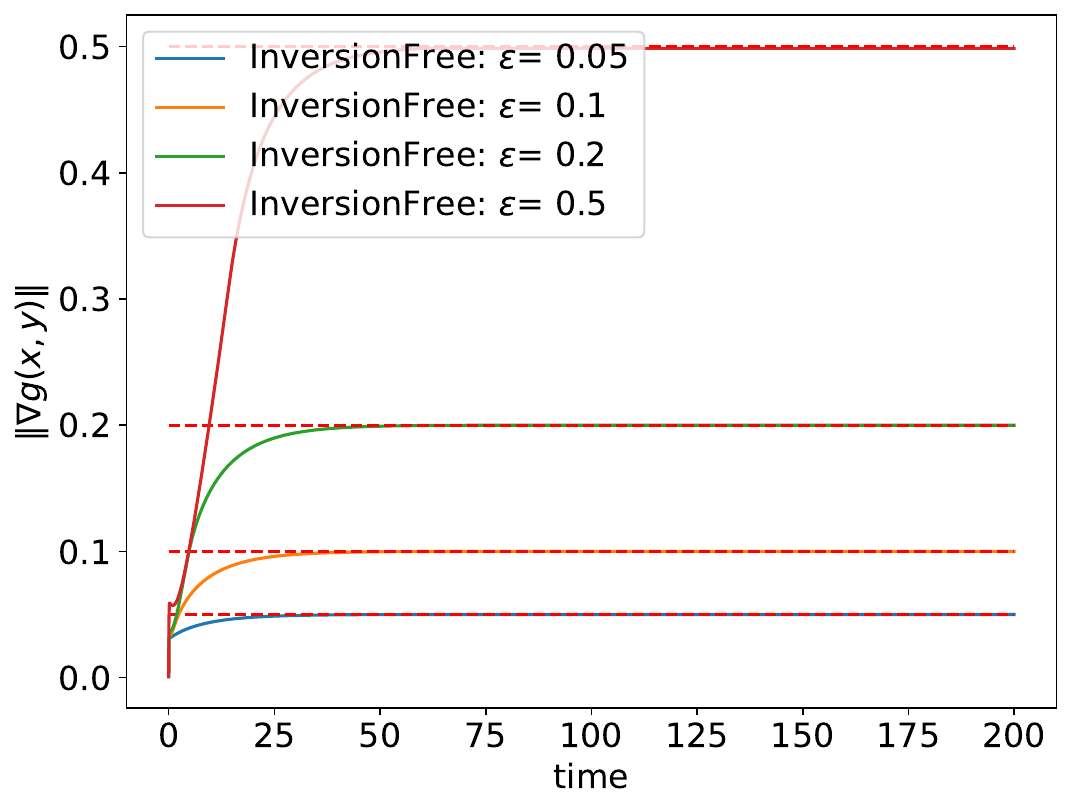}
        % \caption{}
        \label{fig:toyExample-Eps2}
    \end{minipage}
    \begin{minipage}{0.33\textwidth}
        \centering
        \includegraphics[width=\textwidth]{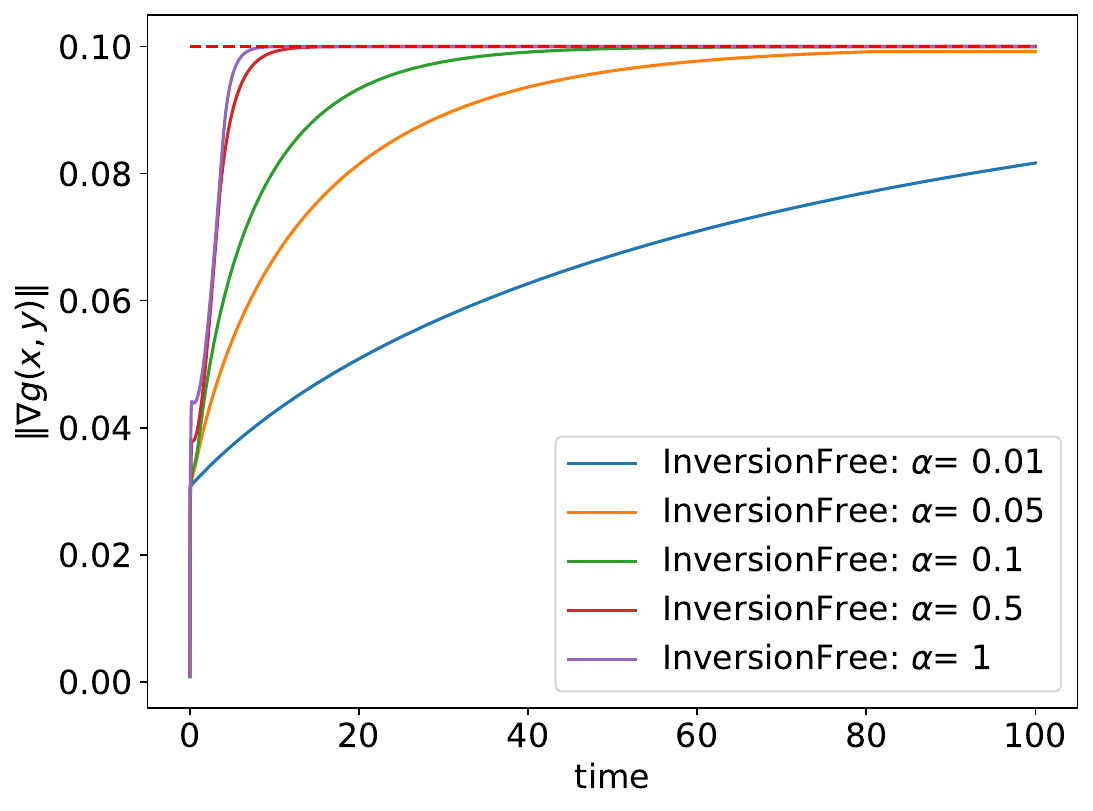}
        % \caption{}
        \label{fig:toyExample-Alpha}
    \end{minipage}%
    \caption{Effect of $\varepsilon$ on the convergence of the (\emph{left:}) surrogate map and (\emph{middle:}) lower-level problem with $\alpha = 0.01$.  
    (\emph{right:}) Effect of $\alpha$ on lower-level behavior. }
    \label{fig:ToyExample}
\end{figure*}

\subsection{Synthetic Example}
    Consider the following basic bilevel optimization problem 
    \begin{align}
            &\min_x \quad  \sin(c^\top x + d^\top y^\star(x)) +  \log(\|x+y^\star(x)\|^2 + 1) \quad \text{s.t.} \quad  y^\star(x) \in \argmin_y  \tfrac{1}{2} \|Hy - x\|^2,\notag
    \end{align}
    where $x, y, c, d \in \mathbb{R}^{20}$ and $H \in \mathbb{R}^{20 \times 20}$ is randomly generated such that its condition number is no larger than 10.
    In \Cref{fig:ToyExample} we illustrate the effect of the hyper-parameters $\varepsilon$ and $\alpha$.
    It can be seen that $\alpha$ controls how fast the trajectories approach the boundary of $\{(x,y) \mid \|\nabla_y g(x,y)\| \leq \varepsilon\}$.

\begin{figure}[t]
    \centering
        \includegraphics[width=0.35\textwidth]{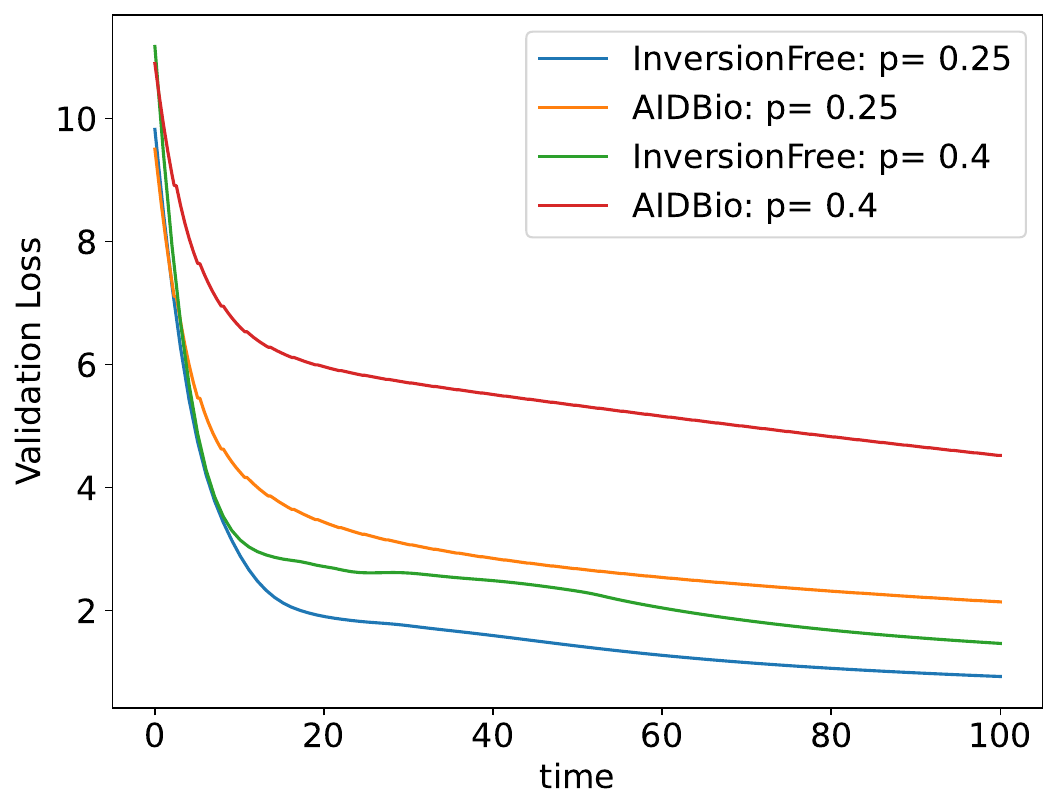}
        \includegraphics[width=0.35\textwidth]{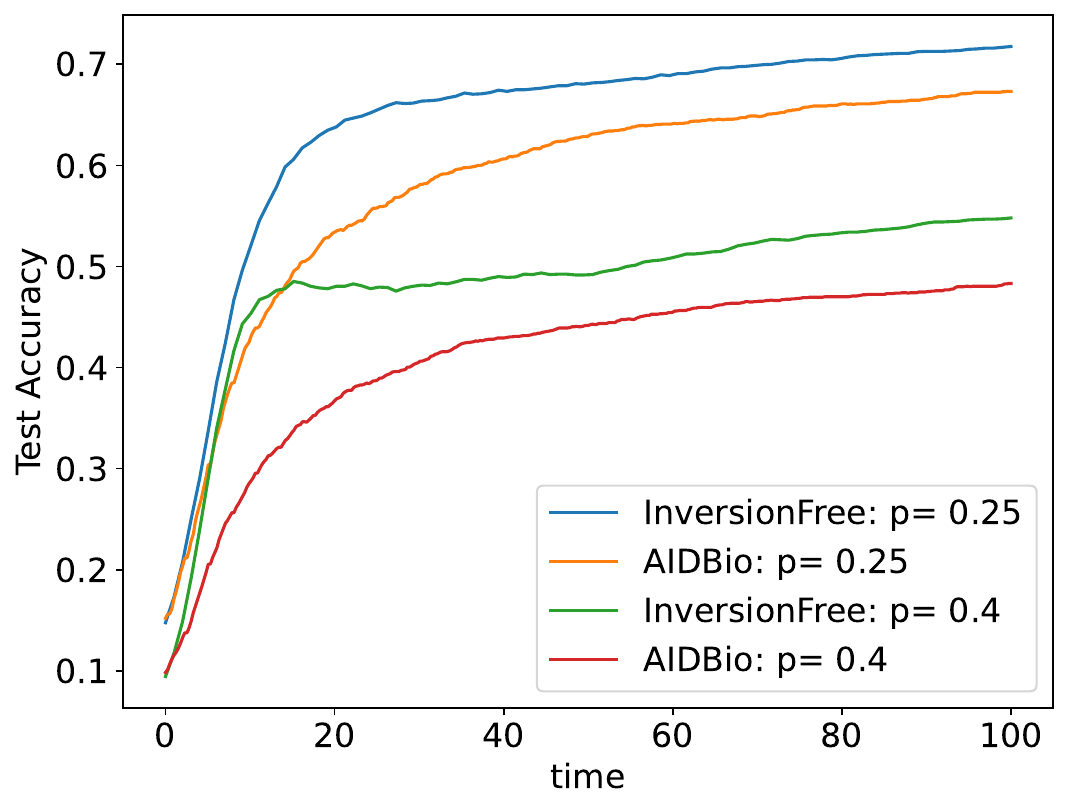}
    \caption{
    The comparison of the validation loss and the test accuracy between our Inversion-free method and AIDBio with $\alpha, \beta \in \{0.001, , 0.01, 0.1\}$ for $p \in \{ 25\%$, $40\% \}$.
    }
    \label{fig:dhc-results}
\end{figure}
    
\subsection{Data Hyper-Cleaning on MNIST}
    Consider the data hyper-cleaning problem,
    where some of the labels in the training data have been corrupted, and the goal is to train a classifier utilizing the clean validation data.
    The objective function is given by
    \begin{align}
        &\min_x \frac{1}{N_{\text{val}}} \sum_{(a_{i}, b_{i}) \in \mathcal{D}_{\text{val}}} \mathcal{L}(a_{i}^\top y^\star(x), b_{i}) \notag 
        &\text{s.t.} \  y^\star(x) \in \argmin_y
        \frac{1}{N_{\text{tr}}} 
        \sum_{(a_{i}, b_{i}) \in \mathcal{D}_{\text{tr}}} \sigma(x_i) \mathcal{L}(a_{i}^\top y, b_{i}) + \! \lambda  \|y\|^2, \notag
    \end{align}
    where $\lambda = 0.001$ is the regularizer and $\sigma(.)$ and $\mathcal{L}(.)$ represent the sigmoid function and cross-entropy loss, respectively.
    Due to the high dimensions of MNIST, we use PCA (maintaining $90 \%$ of the variance) to reduce the dimensions of the problem to $y \in \mathbb{R}^{82 \times 10}$ and $x \in \mathbb{R}^{5000}$, and split the data into training $\mathcal{D}_{\text{tr}}$, validation $\mathcal{D}_{\text{val}}$, and test $\mathcal{D}_{\text{test}}$. 
    We then randomly corrupt the label of $p \%$ of the classes. 
    The goal is to identify the wrong labels in $\mathcal{D}_{\text{tr}}$ using the clean data of $\mathcal{D}_{\text{val}}$.
    % Furthermore $N_{\text{tr}}$ and $N_{\text{val}}$ denote the number of samples in the training and validation dataset.
    Here, the inner problem finds the optimal classifier weight $y^\star(x)$, and the outer problem finds the optimal sample weights $x^\star$ minimizing the validation loss.
    \Cref{fig:dhc-results} show the validation loss and test accuracy of our model compared to AIDBio \cite{ji2021bilevel}, where we observe that our method outperforms AIDBio. 
    To have a fair comparison, we counted the number of gradient evaluations in RK-4 and implemented AID-Bio with the same number of gradient evaluations.

\section{Conclusion and Future Works}

We introduced a control-theoretic approach to solving bilevel optimization problems. By combining a gradient flow mechanism for minimizing the upper-level objective with a safety filter to enforce the constraints induced by the lower-level problem, we developed a single-loop safe gradient flow capable of addressing bilevel problems effectively. To further enhance scalability with respect to lower-level problem dimensions, we proposed a relaxed formulation that maintains the lower-level solution within a user-defined distance while minimizing the upper-level objective. Using Lyapunov analysis, we provided theoretical guarantees for the convergence of both the standard and relaxed safe gradient flows, demonstrating convergence to a neighborhood of the optimal solution. Numerical experiments validated the practicality and robustness of the proposed methods in various scenarios. These contributions offer a novel perspective on bilevel optimization, bridging control theory and optimization. Future work will explore extensions to more complex problem settings, iterative implementation strategies, and further relaxation of assumptions to broaden the applicability of the proposed framework.

\clearpage
\appendix
\appendixpage

\section{Additional Material and Omitted Proofs}\label{Appendix:add}

\subsection{Remarks of \Cref{sec:Invariance}:}
\begin{lemma}\label{rm:Lip}
    The right-hand side of \eqref{eq: gradient flow upper level safe} is Lipschitz continuous, guaranteeing that the ODE is well-posed and admits a unique solution.
\end{lemma}
\begin{proof}
    Equation \eqref{eq:convex QP problem} defines a strongly convex optimization problem, which guarantees that any solution, if it exists, is unique. Furthermore, the parametric Quadratic Program \eqref{eq:convex QP problem} can be viewed as a parametric projection map \(\text{Proj}_{\mathcal{C}(x,y)}(v(x,y))\), where \(v(x,y) = [-\nabla_x f(x,y), -\nabla_y f(x,y)]\), and \(\mathcal{C}(x,y)\) contains \(m\) linear equality constraints defined by the matrix \(\begin{bmatrix} \nabla_{yx}^2 g(x, y) & \nabla_{yy}^2 g(x, y) \end{bmatrix}\), and the vector \(b(x, y) = \alpha \nabla_y g(x, y)\). Under the assumptions that the gradients \(\nabla_x f\), \(\nabla_y f\), \(\nabla_{yx}^2 g\), and \(\nabla_{yy}^2 g\) are Lipschitz continuous and that \(g(x, y)\) is strongly convex in \(y\), the strong convexity ensures that 
    \(\begin{bmatrix} \nabla_{yx}^2 g(x, y) & \nabla_{yy}^2 g(x, y) \end{bmatrix}\)
    has full row rank, satisfying the Linear Independence Constraint Qualification (LICQ). Additionally, the rank of the matrix remains constant in a neighborhood around the solution, satisfying the Constant Rank Condition. Therefore, by \cite[Theorem 2]{ralph1995directional}, the solution map \((x, y) \mapsto (\dot{x}, \dot{y})\) is locally Lipschitz continuous.
\end{proof}

\begin{lemma}\label{rm:Equil}
    The equilibrium of \eqref{eq: gradient flow upper level safe} satisfies the KKT conditions of \eqref{eq: bilevel task 1 b}.
\end{lemma}

\begin{proof}
    The KKT conditions of \eqref{eq: bilevel task 1 b}
    \begin{align}
        \begin{cases}
            \nabla_{x} \mathcal{L} = \nabla_x f(x,y) + \nabla_{yx}^2 g(x,y) \lambda = 0 \\
            \nabla_{y} \mathcal{L} =  \nabla_y f(x,y) + \nabla_{yy}^2 g(x,y) \lambda = 0 \\
            \nabla_y g(x,y) = 0 \\
        \end{cases} \notag
    \end{align}
    is equivalent to the KKT condition of \eqref{eq:convex QP problem} at the equilibrium ($\dot{x}, \dot{y} = 0$) 
    \begin{align}
        \begin{cases}
            \nabla_{\dot{x}} \mathcal{L} = \dot{x} + \nabla_x f(x,y) + \nabla_{yx}^2 g(x,y) \lambda = 0 \\
            \nabla_{\dot{y}} \mathcal{L} = \dot{y} + \nabla_y f(x,y) + \nabla_{yy}^2 g(x,y) \lambda = 0 \\
            \nabla_{yx}^2 g(x,y) \dot{x} + \nabla_{yy}^2 g(x,y) \dot{y} + \alpha \nabla_y g(x,y) = 0\\
        \end{cases} \notag
    \end{align}
    and the proof is complete.
\end{proof}

\begin{lemma}\label{lem:error-grad-ell-fo} 
Suppose Assumption \ref{assumption:lowerlevel} holds. For any vector $d\in\mathbb R^n$,  define  $A(x,y,d) :=  \nabla_x f(x,y) - \nabla_{yx}^2 g(x,y)^\top d $ and $B(x,y,d) :=  \nabla_y f(x,y) - \nabla_{yy}^2 g(x,y) d $, then we have
\begin{align}\label{eq: err_grad_ell}
    \|F(x, y)\| \leq \| A(x,y,d)\| + \frac{L_{yx}^g}{\mu_g} \| B(x,y,d)\|.
\end{align}
\end{lemma}
\begin{proof}
     Recall the definition of $F(x, y)$ in \eqref{eq:grad-surrogate}. Adding and subtracting the term $\nabla_{yx}^2 g(x,y)^\top d$ followed by using the triangle inequality yields
    \begin{align}
      \|F(x, y)\| &= \| \nabla_x f(x,y) - \nabla_{yx}^2 g(x,y)^\top[\nabla_{yy}^2 g(x,y)]^{-1}\nabla_y f(x,y) \|  \notag\\
      &= \|\nabla_x f(x,y) \pm \nabla_{yx}^2 g(x,y)^\top d  -\nabla_{yx}^2 g(x,y)^\top[\nabla_{yy}^2 g(x,y)]^{-1}\nabla_y f(x,y) \|  \notag \\
      &\leq \| A(x,y,d) \| + \|\nabla_{yx}^2 g(x,y)^\top[\nabla_{yy}^2 g(x,y)]^{-1}\Big( \notag \nabla_{yy}^2 g(x,y) d  - \nabla_y f(x,y) \Big) \|. \nonumber
    \end{align}
    Next, combining Assumption~\ref{assumption:lowerlevel} along with the application of the Cauchy–Schwarz inequality we have \eqref{eq: err_grad_ell}.
\end{proof}

\subsection{Proof of \Cref{thm:Original}:}\label{proof:thmSGF}
\begin{proof}
To prove the first claim, we start from the condition given in \eqref{eq: lower level invariance unconstrained}:
\[
\frac{d}{dt} \nabla_y g(x(t), y(t)) + \alpha \nabla_y g(x(t), y(t)) = 0.
\]
Solving this linear differential equation yields:
\[
\nabla_y g(x(t), y(t)) = e^{-\alpha t} \nabla_y g(x(0), y(0)),
\]
which implies:
\[
\|\nabla_y g(x(t), y(t))\| = e^{-\alpha t} \|\nabla_y g(x(0), y(0))\|.
\]
Next, we proceed to the second claim. By Lemma \ref{lem:error-grad-ell}, we have:
\[
\|\nabla \ell(x)\| \leq \|F(x, y)\| + M_1 \|y - y^\star(x)\|.
\]
Using the strong convexity of \( g(x, y) \) (Assumption \ref{assumption:lowerlevel}) and \cite[Theorem 2.1.10]{nesterov2018lectures}, we know:
\[
\mu_g \|y - y^\star(x)\| \leq \|\nabla_y g(x, y)\|.
\]
Thus
\[
\|\nabla \ell(x)\| \leq \|F(x, y)\| + \frac{M_1}{\mu_g} e^{-\alpha t} \|\nabla_y g(x(0), y(0))\|.
\]
We use Young's inequality
\begin{align}\label{eq:young'sIneq}
(a + b)^2 \leq (1 + \gamma)a^2 + \left(1 + \frac{1}{\gamma}\right)b^2, \quad \forall \gamma > 0,
\end{align}
with $\gamma = 1$, we obtain
\begin{align}\label{eq:bound on ell}
\|\nabla \ell(x)\|^2 \leq  2\|F(x, y)\|^2 + 
2\big( \frac{M_1}{\mu_g} e^{-\alpha t} \|\nabla_y g(x(0), y(0))\|\big)^2 .
\end{align}

Now, consider the energy function \( \mathcal{E}(t) \) defined in the theorem. Taking its time derivative gives:
\[
\dot{\mathcal{E}}(t) 
= \underbrace{\nabla_x f(x, y)}_{= -\dot{x} - \nabla^2_{yx} g^\top \lambda}^\top \dot{x} + \underbrace{\nabla_y f(x, y)}_{= -\dot{y} - \nabla^2_{yy} g^\top \lambda}^\top \dot{y} 
+ \beta \underbrace{\big(\nabla^2_{yx} g(x, y) \dot{x} + \nabla^2_{yy} g(x, y) \dot{y}\big)^\top}_{= -\alpha \nabla_y g(x, y)} \frac{\nabla_y g(x, y)}{\|\nabla_y g(x, y)\|} 
+ c \|F(x, y)\|^2.
\]
Simplifying yields
\[
\dot{\mathcal{E}}(t) 
= - \| \dot{x}\|^2 - \|\dot{y}\|^2 - \alpha \beta \|\nabla_y g(x, y)\| + \alpha \lambda(x, y)^\top \nabla_y g(x, y) + c \|F(x, y)\|^2.
\]
We proceed by using Lemma \ref{lem:error-grad-ell-fo} with $d = -2 \lambda(x,y) \nabla_y g(x,y)$, which results in
$\|F(x, y)\| \leq \| \dot{x}\| + \frac{L_{yx}^g}{\mu_g} \| \dot{y}\|$. By applying Young's inequality, we have
\[
\|F(x, y)\|^2 \leq (1 + \gamma)\| \dot{x}\|^2 + \left(1 + \frac{1}{\gamma}\right)\frac{(L_{yx}^g)^2}{\mu_g^2} \| \dot{y}\|^2.
\]
Choosing \( \gamma = \frac{(L_{yx}^g)^2}{\mu_g^2} \), the bound simplifies to
\[
\|F(x, y)\|^2 \leq \left(1 + \frac{(L_{yx}^g)^2}{\mu_g^2}\right)\| \dot{x}\|^2 + \left(1 + \frac{(L_{yx}^g)^2}{\mu_g^2}\right) \| \dot{y}\|^2.
\]
Substituting this back into \( \dot{\mathcal{E}}(t) \), we obtain
\[
\dot{\mathcal{E}}(t) \leq \big(c \big(1 + \frac{(L_{yx}^g)^2}{\mu_g^2}\big) - 1\big)(\| \dot{x}\|^2 + \| \dot{y}\|^2) 
+ \alpha(\lambda(x, y)^\top \nabla_y g(x, y) - \beta \|\nabla_y g(x, y)\|).
\]
The dual multiplier $\lambda(x,y)$ defined in \eqref{eq:lambda1} can be bounded as follows:
\begin{align}\label{eq:dual-bound}
\|\lambda(x,y)\| &\leq \frac{1}{\mu_g^2} (C_{yx}^g C_x^f + C_{yy}^g C_y^f+\alpha e^{-\alpha t}\|\nabla_y g(x(0),y(0))\|)\nonumber \\
&\leq \beta:=\frac{1}{\mu_g^2} (C_{yx}^g C_x^f + C_{yy}^g C_y^f+\alpha \|\nabla_y g(x(0),y(0))\|),
\end{align}%
where we used the strong convexity of $g$, combined with the boundedness of all terms involved in \eqref{eq:lambda1} according to Assumptions \ref{assumption:lowerlevel} and \ref{assumption:upperlevel}.
Choosing \( c := \frac{\mu_g^2}{\mu_g^2 + (L_{yx}^g)^2} \) and using \eqref{eq:dual-bound}, it follows that
$
\dot{\mathcal{E}}(t) \leq 0,
$
which implies \( \mathcal{E}(t) \leq \mathcal{E}(0) \).
Finally, since \( f(x(t), y(t)) - f^\star \geq 0 \), we obtain
\begin{align}
\int_{0}^t \|F(x(\tau), y(\tau))\|^2 d\tau 
&\leq \frac{1}{c} \mathcal{E}(0). \notag
\end{align}
Combining this with the bound on \( \|\nabla \ell(x)\|^2 \) in \eqref{eq:bound on ell}, 
$$
\int_{0}^t \|\nabla \ell (x(\tau))\|^2 d\tau \leq \frac{2\mathcal{E}(0)}{c} + 
2 \int_{0}^t \big(\frac{M_1}{\mu_g} e^{-\alpha \tau} \|\nabla_y g(x(0), y(0))\|\big)^2 d\tau
$$
completes the proof.
\end{proof}

\subsection{Derivation of \eqref{eq:projectionXY}:}\label{derivation}
    Writing the Lagrangian and the KKT conditions, we obtain
    \begin{align}\label{eq:KKTrelaxed}
        \mathcal{L}(\dot{y}, \dot{x}, \lambda) \! &= \! 
        \frac{1}{2} \|\dot{x} \! + \! \nabla_x f(x,y)\|^2 \! + \! \frac{1}{2} \|\dot{y} \! + \! \nabla_y f(x,y)\|^2 \! + \!  
        \lambda (\nabla_x h(x,y)^\top \dot{x} \! + \! \nabla_y h(x,y)^\top \dot{y} \! + \! \alpha (h(x,y) \! - \! \varepsilon^2)), \notag \\
        &\qquad \begin{cases}
            \nabla_{\dot{x}} \mathcal{L} = (\dot{x} + \nabla_x f(x,y)) + \lambda \nabla_x h(x,y) = 0 \\
            \nabla_{\dot{y}} \mathcal{L} = (\dot{y} + \nabla_y f(x,y)) + \lambda \nabla_y h(x,y) = 0 \\
            \lambda(\nabla_x h(x,y)^\top \dot{x} + \nabla_y h(x,y)^\top \dot{y} + \alpha (h(x,y) - \varepsilon^2)) = 0 \\
            \nabla_x h(x,y)^\top \dot{x} + \nabla_y h(x,y)^\top \dot{y} +\alpha (h(x,y) - \varepsilon^2)) \leq 0 \\
            \lambda \geq 0
        \end{cases}
    \end{align}
We then obtain $\dot{x} = -\nabla_x f(x,y) - \lambda \nabla_x h(x,y)$ and $\dot{y} = -\nabla_y f(x,y) - \lambda \nabla_y h(x,y)$. 
Substituting these into the rest of the conditions allows us to solve for $\lambda$ as follows
\begin{align}
    \lambda \Big(\nabla_x h(x,y)^\top (-\nabla_x f - \lambda \nabla_x h(x,y)) + 
    \nabla_y h(x,y)^\top (-\nabla_y f - \lambda \nabla_y h(x,y)) 
    +\alpha (h(x,y) - \varepsilon^2)\Big) = 0. \notag
\end{align}

\subsection{Proof of Proposition \ref{thm:epsilon feasible}:}\label{proof:PropRXGF}
\begin{proof}
Proof of \ref{prop:Lip}: 
Let us denote the feasible set of \eqref{eq: invariance condition} with $\mathcal{C}(x,y)$. Then, the quadratic program in \eqref{eq: first-order lower-level dynamics 2} can be viewed as a parametric projection map $\text{Proj}_{\mathcal C(x,y)}(v(x,y))$ where $v(x,y)=[-\nabla_x f(x,y),~-\nabla_y f(x,y)]$. Note that $\mathcal C(x,y)$ is a half-space (contains only one inequality) in which Slater's condition holds. 
Therefore, according to \cite[Theorem 2]{ralph1995directional} combined with Lipschitz continuity of $h$, $\nabla_x h$, $\nabla_y h$, $\nabla_x f$, and $\nabla_y f$, we conclude that $\text{Proj}_{\mathcal C(x,y)}(v(x,y))$ is a locally Lipschitz continuous map. Furthermore, since $(x,y)$ lies in the compact set $L^-_\epsilon(h)$ 
%$\cup \mathbb{B}_r((x(0),y(0)))$ where $r:=\|(x(0),y(0))-(x^*,y^*)\|$, 
we can conclude that the projection map is Lipschitz continuous, hence, the uniqueness of the solution to the ODE follows from Picard-Lindel\"{o}f theorem \cite{haddad2008nonlinear}.

Proof of \ref{prop:Feas}: If \(((x(0), y(0)) \in L_{\varepsilon^2}^{-}(h) \), then $h(x(0),y(0)) -\varepsilon^2 \leq 0$.
This condition coupled with \eqref{eq: invariance condition} ensures that $h(x(t),y(t)) -\varepsilon^2 \leq 0$ for all $t \geq 0$.
% \end{proof}

% \subsubsection*{Proof of \ref{prop:KKT}:}
% \begin{proof}
Proof of \ref{prop:KKT}:
    The KKT conditions of \eqref{eq:bilevel-approx} is equivalent to the KKT condition of \eqref{eq: first-order lower-level dynamics 2} (see \eqref{eq:KKTrelaxed}), at the stationary point $\begin{bmatrix}\dot{x} \\ \dot{y}\end{bmatrix} = \begin{bmatrix}0 \\ 0\end{bmatrix}$.
    \begin{align}
        \begin{cases}
            \nabla_{x} \mathcal{L} = \nabla_x f(x,y) + \lambda \nabla_x h(x,y) = 0 \\
            \nabla_{y} \mathcal{L} =  \nabla_y f(x,y) + \lambda \nabla_y h(x,y) = 0 \\
            \tilde{\lambda}(h(x,y)-\varepsilon^2) = 0 \\
            h(x,y)-\varepsilon^2 \leq 0 \\
            \tilde{\lambda} \geq 0
        \end{cases} \notag
    \end{align}
    
Proof of \ref{prop:feas}: Two cases may occur, if $h(x,y)\leq \epsilon^2$, then $(\dot{x}_d,\dot{y}_d)=(0,0)$ is a trivial feasible point. Otherwise consider $(\dot{x}_d,\dot{y}_d)=(0,-\frac{\alpha(h-\epsilon^2)}{\nabla_y h^\top \nabla_y g}\nabla_y g)$.
\end{proof}

\subsection{Proof of \Cref{thm:converganceRate}:}\label{proof:thmRXGF}
\begin{lemma}[Lemma 2.2.7 \cite{nesterov2018lectures}]\label{lemma:y_dot}
    Let $\mathcal{D}\subseteq \mathbb R^n$ be a nonempty, closed convex set. Then for any $v\in\mathbb R^n$, $w=\mathrm{Proj}_{\mathcal{D}}(v)$ if and only if 
    %\begin{equation*}
        $(v-w)^\top(u-w)\leq 0$, $\forall u\in\mathcal D$.
    %\end{equation*}
\end{lemma}

\begin{lemma}\label{lemma:safeDynamics}
      If $(x, y) \in L^{-}_\varepsilon(h)$,
      for any $\dot{x}$, the choice of
    $
    \dot{y} = - [\nabla_{yy}^2 g(x,y)]^{-1} \nabla_{yx}^2 g(x,y) \dot{x}
    $
    satisfies \eqref{eq: invariance condition}.
\end{lemma}

\begin{proof}
    From \( (x, y) \in L^{-}_\varepsilon(h)\), we have
    \(\alpha (h(x,y) - \varepsilon^2) \leq 0\). Thus we can write
    \begin{align}
        \underbrace{\nabla_x h(x,y)^\top \dot{x} \!-\! \nabla_y h(x,y)^\top {\nabla_{yy}^2 g(x,y)}^{-1} \nabla_{yx}^2 g(x,y) \dot{x}}_{=0}  +\alpha (h(x,y) - \varepsilon^2) \leq 0\notag.
    \end{align}
    where we use $\nabla_x h(x,y) = 2 \nabla_{yx}^2g(x,y)^\top \nabla_y g(x,y)$ and
    $\nabla_y h(x,y) = 2 \nabla_{yy}^2g(x,y)^\top \nabla_y g(x,y)$ by definition.
\end{proof}

\begin{proof}
    Using Lemma \ref{lem:error-grad-ell}, we have
    \begin{align}\label{eq:boundEll}
        \|\nabla \ell(x)\| &\leq \|F(x,y)\| + M_1\|y - y^\star(x)\| \leq \|F(x,y)\| +  \frac{M_1}{\mu_g}\varepsilon,
    \end{align}
    where the second inequality is a consequence of the strong convexity of $g(x,y)$ according to Assumption \ref{assumption:lowerlevel}, where we invoke \cite[Theorem 2.1.10]{nesterov2018lectures} to conclude that
    %\begin{align}
    % \label{eq:NesterovConvex}
        $\mu_g \|y \!-\! y^\star(x)\| \leq 
        % \!  \! \|\nabla_y g(x,y) \!-\! \nabla_y g(x,y^\star(x))\| \! 
        \! \|\nabla_y g(x,y)\|$. 
        %\notag
    %\end{align}
    Now we use \eqref{eq:young'sIneq} with $\gamma = 1$ and obtain 
    \begin{align}\label{eq: bound ell with young}        
    \|\nabla \ell(x)\|^2 \leq 2\|F(x,y)\|^2 + 2\frac{M_1^2}{\mu_g^2}\varepsilon^2.
    \end{align}
    
    Then, consider the function defined in \eqref{eq:lyapunov-inversion-free}. Taking the derivative w.r.t to $t$ followed by adding and subtracting $\Tilde{d_1}(x(t),y(t))^\top \dot{x}(t)$ imply that
        \begin{align}\label{prf:dynamic}
            \dot{\mathcal{E}}(t) 
            % & =\nabla_x f(x(t),y(t))^\top \dot{x}(t) +  \nabla_y f(x(t),y(t))^\top \dot{y}(t)  \notag\\
            % & \quad + c \|F(x(t), y(t))\|^2 \nonumber \\
            & = \Big(\nabla_x f(x(t),y(t)) +  \Tilde{d}_1(x(t),y(t))\Big)^\top \dot{x}(t) - \Tilde{d}_1(x(t),y(t))^\top \dot{x}(t) \nonumber\\
            & \quad +  \nabla_y f(x(t),y(t))^\top \dot{y}(t)+ c \|F(x(t), y(t))\|^2 \nonumber \\
            & = - \| \dot{x}(t)\|^2 - \Tilde{d}_1(x(t),y(t))^\top \dot{x}(t) + \underbrace{\nabla_y f(x(t),y(t))^\top \dot{y}(t)}_{\kappa} + c \|F(x(t), y(t))\|^2,
        \end{align}
    where we define $\tilde{d}_1(x, y) = \lambda(x,y) \nabla_x h(x,y)$ and $\tilde{d}_2(x, y) = \lambda(x,y) \nabla_y h(x,y)$.
    Next, we provide an upper bound for $\kappa$ using \Cref{lemma:y_dot}. 
    In particular, define
    \begin{align}
        v &:= \begin{bmatrix}-\nabla_x f(x,y)\\ -\nabla_y f(x,y)\end{bmatrix}, 
        \quad w := \begin{bmatrix} \dot{x} \\ \dot{y}\end{bmatrix}, \notag \\
        u &:= \begin{bmatrix}
            \dot{x}\\- [\nabla_{yy}^2 g(x(t),y(t))]^{-1} \nabla
        _{yx}^2 g(x(t),y(t)) \dot{x}(t)
        \end{bmatrix}.\notag
    \end{align}
    Note that according to \eqref{eq: first-order lower-level dynamics 2}, $w$ is the projection of $\nu$ onto the constraint set. 
    Then using \eqref{eq: invariance condition} and that $(x(t), y(t)) \in L_{\varepsilon^2}^{-}(h) \ \forall t \geq 0$ we conclude that $u$ satisfies the invariance condition, based on Lemma \ref{lemma:safeDynamics}.
    Then according to Lemma \ref{lemma:y_dot}, one can write $(v - w)^\top(u - w) \leq 0$
    \begin{align}\label{prf: y_dot_p}
        \nabla_y f(x(t), y(t))^ \top \dot{y}(t) 
        &\leq \tilde{d}_2^\top \nabla_{yy}^2 g (x(t), y(t))^{-1} \nabla_{yx}^2 g(x(t), y(t)) \dot{x}(t) 
        - \| \dot{y}(t)\|^2  \notag \\
        &= \tilde{d}_1^\top(x(t),y(t)) \dot{x}(t) - \| \dot{y}(t)\|^2.
    \end{align}
    Then combining \eqref{prf:dynamic} and \eqref{prf: y_dot_p} leads to 
    \begin{align}
        \dot{\mathcal{E}}(t) &\leq - \| \dot{x}(t)\|^2 - \| \dot{y}(t)\|^2 + c \|F(x, y)\|^2 \nonumber
         \leq (c (1 + \frac{(L_{yx}^g)^2}{\mu_g^2}) -1) (\| \dot{x}(t)\|^2+\| \dot{y}(t)\|^2), \notag
        %+ \Big( c\big(\frac{L_{yx}^g}{\mu_g}(1 + \frac{L_{yx}^g}{\mu_g})\big) -1\Big) \| \dot{y}(t)\|^2
    \end{align}
    where the last inequality is obtained using \eqref{eq:young'sIneq} and \Cref{lem:error-grad-ell-fo} with $d = \lambda$.
    Choosing $c=\mu_g^2/(\mu_g^2+(L_{yx}^g)^2)$ implies that $\dot{\mathcal{E}}(t) \leq 0$, hence, $\mathcal{E}(t)\leq \mathcal{E}(0)$. 
    Then, using the fact that $f(x(t),y(t)) - f^\star_\varepsilon\geq 0$ and dividing both sides of the inequality by $t$ we obtain
    $
        \frac{1}{t}\int_{0}^t \|F(x(\tau), y(\tau))\|^2 d\tau\leq \frac{1}{ct}( f(x(0),y(0)) - f^*_{\varepsilon}). \notag
    $
    Finally, combining this result with \eqref{eq: bound ell with young} finishes the proof.

\end{proof}

\section{Prediction-Correction Dynamics}\label{sec:secondOrder}
Another approach to analyzing \eqref{eq: bilevel task 1 b} is to allow $\dot{x}$ to follow the gradient flow for minimizing the implicit objective function, $\dot{x} = -\nabla \ell(x)$. 
According to \eqref{eq:grad-implicit}, calculating $\nabla \ell(x)$  requires an exact lower-level solution $y^\star(x)$. Thus, we instead consider the following common surrogate map to estimate \eqref{eq:grad-implicit} by replacing the optimal lower-level solution with an approximated solution $y$,
$$
\dot{x} = -F(x,y). 
$$
With the upper-level dynamics set, we focus on designing the lower-level dynamics.
The lower-level objective is indexed by the upper-level decision variable $x$. 
We propose to view this decision variable as a dynamic parameter $x(t)$ that evolves depending on the dynamics of the upper-level optimizer. 
Given such a trajectory, our main idea is then to design a dynamical system for the lower-level optimizer that tracks the optimal \emph{trajectory} $y^\star(t):= y^\star(x(t))$. 
To this end, we consider the following \emph{prediction-correction} update law inspired by \cite{fazlyab2017prediction},
\begin{align}
    \dot{y} &= - [\nabla_{yy}^2 g(x,y)]^{-1}  [\beta \nabla_y g(x,y)+\nabla_{yx}^2g(x,y) \dot{x}], \notag
\end{align}
where $\beta>0$. This dynamics consists of a Newton's correction term ($- [\nabla_{yy}^2 g(x,y)]^{-1}\nabla_y g(x,y)$) that contracts $y$ towards $y^\star(x)$, and a prediction term ($- [\nabla_{yy}^2g(x,y)]^{-1}  \nabla^2_{yx}g(x,y) \dot{x}$) that tracks the changes in $y^\star(x)$ as $x$ evolves. The overall dynamics ensures $y(t)$ converges to $y^\star(x(t))$ exponentially fast (see \Cref{lem:y-rate}). 
Based on the preceding discussion,  we propose the following update law for the variables $x$ and $y$,
\begin{subequations}\label{eq:dynamic-x-so}
\begin{align}
    \dot{x} &= - F(x,y), \label{eq:dynamic-x}\\
    \dot{y} &= - [\nabla_{yy}^2 g(x,y)]^{-1}  [\beta \nabla_y g(x,y)+\nabla_{yx}^2g(x,y) \dot{x}]. \label{eq:dynamic-y}
    %\dot{y}(t) &= - H(t)  [\alpha \nabla_y g(x(t),y(t))+\nabla_{xy}^2g(x(t),y(t)) \dot{x}(t)],
\end{align}
\end{subequations}

\subsection{Convergence Analysis}
To analyze the convergence of the proposed prediction-correction bilevel ODE \eqref{eq:dynamic-x-so}, we first show that the lower-level trajectory $y(t)$ is globally exponentially convergent to the lower-level optimal trajectory $y^\star(x(t))$. 
\begin{proposition}[Contraction of Lower-level Dynamics]\label{lem:y-rate}
    Let $(x(t),y(t))$ be the solution of \eqref{eq:dynamic-x-so}. Suppose Assumption \ref{assumption:lowerlevel} holds. Then for any $t\geq 0$,
    \begin{align}
        \|y(t) - y^\star(x(t))\|\leq \frac{1}{\mu_g}\|\nabla_y g(x(0),y(0))\| e^{-\beta t}.
    \end{align}
\end{proposition}
\begin{proof}
The proof follows a similar approach to \cite{fazlyab2017prediction}, and is omitted here for brevity.
\end{proof}

To analyze the convergence rate of the proposed ODE \eqref{eq:dynamic-x-so}, we define the following Lyapunov function,
\begin{align}\label{eq:Lyapanov}
    \mathcal{E}(t) &= \ell(x(t))-\ell(x^\star) + \frac{1}{2}\int_{0}^t \|\nabla \ell(x(\tau))\|^2 d\tau, %\frac{1}{2} \|\nabla_y g(x(t),y(t))\|^2
\end{align}
which is the optimality residual for the single-level variant of problem \eqref{eq: bilevel task 1} evaluated along the trajectories $x(t)$. 

\begin{theorem} \label{thm:convergence-rate-pred-corr}
    Let $x(t)$ and $y(t)$ be the solutions of \eqref{eq:dynamic-x-so}. Suppose Assumption \ref{assumption:lowerlevel} holds. Then for any arbitrary initialization $(x(0),y(0))$, and  $\beta>0$, 
    \begin{align}
        \frac{1}{2t}&\int_{0}^t \|\nabla \ell(x(\tau))\|^2 d\tau\leq \frac{1}{t}\left(\ell(x(0))\!-\!\ell(x^\star)+\frac{M_1^2}{4\beta \mu_g^2}\|\nabla_y g(x(0),y(0))\|^2\right). \notag
    \end{align}
\end{theorem}
\begin{proof}
Taking the time derivative of the Lyapunov function \eqref{eq:Lyapanov} along the trajectories of \eqref{eq:dynamic-x-so}, we obtain
\begin{align}
    \dot{\mathcal{E}}(t) \! =
    \! \nabla \ell(x)^\top \dot{x}(t) + \tfrac{1}{2}\|\nabla \ell(x)\|^2 \! = \!
    -\nabla \ell(x(t))^\top (F(x(t),y(t)) \!- \! \nabla \ell(x)) \! - \! \tfrac{1}{2}\|\nabla \ell(x)\|^2. \notag
\end{align}
Now using Young's inequality, we have
$$
 \dot{\mathcal{E}}(t) \leq \tfrac{1}{2}\norm{F(x(t),y(t))-\nabla \ell(x(t))}^2 \nonumber.
$$
Then, by integrating both sides we obtain
\begin{align}
    \mathcal{E}(t) - \mathcal{E}(0)&\leq \int_{0}^t \tfrac{1}{2}\norm{F(x(\tau),y(\tau))-\nabla \ell(x(\tau))}^2 d\tau\nonumber\\
    &\leq \tfrac{M_1^2}{2} \int_{0}^t \norm{y(\tau)-y^\star(x(\tau))}^2 d\tau \nonumber\\
    &\leq \tfrac{M_1^2}{2\mu_g^2} \|\nabla_y g(x(0),y(0))\|^2  \int_{0}^t e^{-2\beta \tau} d\tau \nonumber\\
    &\leq \tfrac{M_1^2}{4\beta \mu_g^2} \|\nabla_y g(x(0),y(0))\|^2, \notag
\end{align}
where in the second inequality we used \Cref{lem:error-grad-ell} and the penultimate inequality follows from \Cref{lem:y-rate}. 
Finally, noting that $\ell(x(t)) \geq \ell(x^\star)$, 
 we conclude the result.
\end{proof}

\begin{remark}[Comparison with \cite{chen2022single}]
    In \cite{chen2022single}, the authors propose an ODE with the same upper-level dynamics as in \eqref{eq:dynamic-x}, but with slightly different lower-level dynamics, where a gradient correction step is used in place of Newton's correction step. However, the computational complexity remains identical as the prediction term already uses the Hessian inverse. 
    %Our numerical experiments will show that this modification improves the convergence rate of the ODE to the optimal solution. 
    Furthermore, \cite{chen2022single} only establishes asymptotic convergence, while we use a different Lyapanouv function that establishes a convergence rate as stated in \Cref{thm:convergence-rate-pred-corr}.
\end{remark}

\begin{figure}[t]
    \centering
        \includegraphics[width=0.4\textwidth]{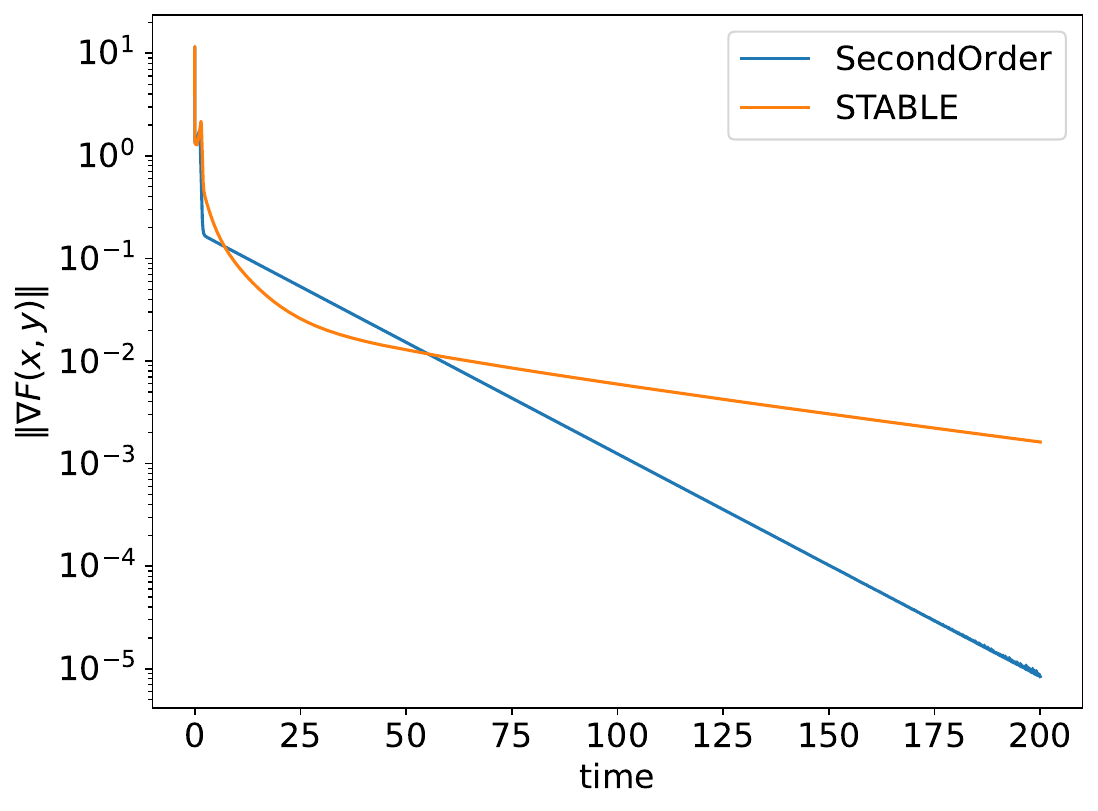}
    \caption{
    Comparison between our second order method and STABLE with $\alpha = 1$ and $\beta \in \{10^{-2}, 5\!\times\!10^{-2}, 10^{-1}, 5\!\times\!10^{-1} \}$.
    }
    \label{fig:stable}
\end{figure}

\Cref{fig:stable} is a comparison between our prediction-correction dynamics designed in \eqref{eq:dynamic-x-so} and continuous-time STABLE \cite{chen2022single}. 
    It can be seen that Newton's correction considerably speeds up the convergence compared to STABLE.

\printbibliography
\end{document}